\documentclass[a4paper,10pt]{article}

\usepackage{indentfirst,amsmath,amsfonts,amstext,amssymb,amscd,bezier,amsthm}
\usepackage[latin1]{inputenc}
\usepackage[dvips]{graphicx}
\usepackage{setspace}
\usepackage{graphicx}
\usepackage{tabularx}
\usepackage{longtable}
\usepackage[usenames,dvipsnames]{pstricks}
\usepackage{epsfig}
\usepackage{pst-grad}
\usepackage{pst-plot}
\usepackage[all]{xy}
\usepackage{hyperref}
\usepackage[T1]{fontenc}
\usepackage{ae,aecompl}
\usepackage{indentfirst}
\usepackage{xspace}
\usepackage{color}
\usepackage{psfrag}
\usepackage{mathrsfs}
\usepackage{upgreek}                            
\usepackage{makeidx}                            
\usepackage{fancyhdr}
\usepackage{fancybox}
\usepackage[rm]{titlesec}
\usepackage{float}        
\usepackage{cases}
\usepackage{enumerate}

\input xy
\xyoption{all}

\newcommand{\p}{\mathbb{P}}

\newcommand{\F}{\mathbb{F}}

\newcommand{\lra}{\longrightarrow}

\newcommand{\A}{\mathcal{A}}

\newcommand{\ff}{\mathcal{F}}

\newcommand{\fq}{\mathbb{F}_q}

\newcommand{\fqm}{\mathbb{F}_{q^m}}
\newcommand{\fqu}{\mathbb{F}_{q^u}}
\newcommand{\fqr}{\mathbb{F}_{q^r}}
\newcommand{\fqc}{\overline{\mathbb{F}}_q}

\newcommand{\D}{\mathcal D}
\newcommand{\xx}{\mathcal X}

\newcommand{\yy}{\mathcal{Y}}
\newcommand{\divi}{\text{div}}
\DeclareMathOperator{\sign}{sign}

\theoremstyle{plain}
\newtheorem{thm}{Theorem}[section]
\newtheorem{defi}[thm]{Definition}
\newtheorem{prop}[thm]{Proposition}

\newtheorem{cor}[thm]{Corollary}
\newtheorem{rem}[thm]{Remark}
\newtheorem{ex}[thm]{Example}

 \font\numberfont= pzcmi scaled
3000
\titleformat{\chapter}[display]
  {\normalfont\Large 
  }
  {
   \filright
   \rule[32pt]{.7\linewidth}{4pt}
   \hspace{-8pt}
   \shadowbox{
   \begin{minipage}{.15\linewidth}
     \begin{center}
          \textsl{\bf {\large \chaptertitlename}}\\
       \vspace{1ex}
       {\bf {\numberfont \thechapter}}\\
       \vspace{1ex}
     \end{center}
   \end{minipage}}
  }
  {-10pt}
  {\filcenter
           \sl
           \bf
              \Huge
     }
  [\vspace{-1cm}\hfill\rule{.8\textwidth}{0.5pt}\\
\vskip-2.8ex\hfill\rule{.7\textwidth}{4pt}\vspace*{-1ex}]
\titlespacing{\chapter}{0pt}{*4}{*1}

\titleformat{\section}[block]
{\normalfont\bfseries} {\thesection}{0.5em}{}

\titleformat{\subsection}[block]
{\normalfont\large\bfseries} {\thesubsection}{0.5em}{}

\makeindex                                      
\makeatletter
\numberwithin{equation}{section}
\begin{document}

\title{Bounds for the number of points on curves over finite fields}

\author{\textbf{Nazar Arakelian} \\
 \small{CMCC, Universidade Federal do ABC, Santo Andr\'e, Brazil} \\
 \textbf{Herivelto Borges}\\
 \small{ICMC, Universidade de S\~ao Paulo, S\~ao Carlos, Brazil}}

\maketitle
 \begin{abstract}
 
Let $\xx$ be a projective irreducible nonsingular algebraic curve defined over a finite field $\fq$.  This paper
presents a variation of the  St\"ohr-Voloch  theory and sets new bounds to the number of $\fqr$-rational points
on $\xx$. In certain cases, where comparison is possible, the results are shown to  improve other bounds such as Weil's,  St\"ohr-Voloch's and Ihara's. 
\end{abstract}

\section {Introduction}

The problem of estimating the number of points on curves over finite fields has engendered a host of  applications. Around 1980, Goppa presented a remarkable application of the theory of curves over finite fields, particularly those with many points, to coding theory. Elliptic-curve cryptography, created by Koblitz and Miller around 1986, is  another notable application. Additional connections can be found in other areas such as finite geometry, combinatorics, and number theory.
%

Let $\xx$ be a projective geometrically irreducible nonsingular  algebraic curve of genus $g$ defined over a finite field $\fq$, and let  $N_r$  denote its  number of $\fqr$-rational points. A celebrated result estimating  $N_r$ is the Weil bound
\begin{equation}\label{hasse-weil}
N_r\leq1+q^r+2g\sqrt{q^r}.
\end{equation}

There are many examples of curves attaining Weil's bound. These are called maximal curves, and they are a significant subject in mathematical research. Nevertheless, in several  cases, Weil's bound has been improved.  In this regard,   noteworthy results were presented by St\"ohr and Voloch in 1986. Their more geometric approach provides bounds  dependent on certain data of an embedding of the curve in $\mathbb{P}^n$. Their method not only proves Weil's bound once again, but offers substantial improvements on it in many circumstances. Accordingly, over the last few decades, researchers have used St\"ohr-Voloch results as an effective tool in addressing a variety of problems related to curves over finite fields.

A fundamental idea in  the St\"ohr-Voloch  method is counting  the number of points $P\in \xx $ whose Frobenius image $P^q \in \xx$ lies on the osculating hyperplane at $P$. The focus of this parper is a natural variation of their approach, which is counting the number  points $P\in \xx$ for which  the line  spanned  by two different  Frobenius images  $P^{q^u}$and $ P^{q^m}$ intersects the $(n-2)$-th  osculating space at $P$. 
The core of the resuts here  paralell those in the original St\"ohr-Voloch paper \cite{SV}. Our bounds, however, will  also  take into account  points defined over exensions of $\fq$, namely, $\fqu$, $\fqm$ and $\F_{q^{m-u}}$ (see Theorem \ref{mainbound}). Such  bounds are  effective in many settings, as will be seen in  Section \ref{examples}.  Camparison with
other bounds  such as Weil's, St\"ohr-Voloch's and Ihara's bound  is sometimes possible and,  in certain cases,   yields   substantial improvement.

This paper is organized as follows. In Section 2, some notation and preliminary results are presented. 
In Section $3$, for any  given pair of  coprime  integers $u$ and $m$, with $m>u\geq 1$, and a nondegenrated  $\F_q$-morphism
$\xx \xrightarrow{\phi } \mathbb{P}^{n}$, we associate an effective  divisor $T_{u,m}$ on $\xx$. Such divisor, which depends only on 
$q$, $u$, $m$ and the  linear series associated to $\phi$, collects the  $\fqr$-rational points on $\xx$,  for $r\in \{1,u,m,m-u\}$.
Section $4$ describes the core of the study.  The weight of each type of point  on the divisor $T_{u,m}$ is estimated and  the principal result  is established.  Some arithmetic properties of a sequence $(\kappa_0,\ldots,\kappa_{n-2})$, that will naturally arise from the data $q$, $u$, $m$ and $\phi$,  are examined. Among other results,   criteria for such sequence to be the classical sequence $(0,\ldots,n-2)$ is provided.
Section $5$ provides examples related to the new bounds and compares the  results with those in the literature, such  as Weil's, St\"ohr-Voloch's and Ihara's bound.
\text{}\\

\textbf{Notation}
\text{}\\
The following notation will be used throughout the paper.

\begin{itemize}
\item $\fq$ is the finite field with $q=p^h$ elements, where $p$ is  a prime, and $\fqc$ denotes its algebraic closure.
\item For an irreducible curve $\yy$ defined over $\fq$ and an algebraic extension $\mathbb{H}$ of $\fq$, the function field of $\yy$ over $\mathbb{H}$ is denoted by $\mathbb{H}(\yy)$.
\item By a point $P \in \yy$, we mean a point in a nonsingular model of $\yy$. More precisely, the points of $\yy$ will be regarded as branches over $\fqc$. 
\item For a curve $\yy$ and $r>0$, the set of its $\F_{q^r}$-rational points is denoted by $\yy(\F_{q^r})$.
\item $N_{r}(\yy)$ is the number of $\F_{q^r}$-rational points of the curve $\yy$. If no confusion arises, $N_r(\yy)$ will be simply denoted by  $N_r$. 
\item For $P \in \yy$, the discrete valuation at $P$ is denoted by $v_P$.
\item Given $g \in \fqc(\xx)$, $t$ a separating variable of $\fqc(\xx)$ and $r \geq 0$, the $r$-th Hasse derivative of $g$ with respect to $t$ is denoted by  $D_t^{(r)}g$.

\end{itemize}


\section{Preliminary results}

Let $\xx$ be a projective geometrically irreducible nonsingular algebraic curve of genus $g$ defined over $\fq$. Associated with the nondegenerated morphism $\phi=(f_0:\cdots:f_n): \xx \lra \p^{n}(\fqc)$, there exists a base-point-free linear series on  $\xx$ of dimension $n$ given by
$$
\D=\left\{\divi\left(\displaystyle\sum_{i=0}^{n}a_if_i\right)+E \ | (\ a_0:\ldots:a_n) \in\mathbb{ P}^n(\fqc)\right\},
$$
where $E:=\sum\limits_{P \in \xx}e_PP$ and $e_P=-min\{v_P(f_0),\ldots,v_P(f_n)\}$. Conversely, each base-point-free linear series   of dimension $n$  on  $\xx$ gives rise to a unique (up to projective transformation) nondegenerated morphism $\xx \lra \p^n(\fqc)$. The degree of the linear series $\D$ is the degree of the divisor $E$. For  each point  $P \in \xx$, there exists a sequence of non-negative integers $j_0(P),\ldots,j_n(P)$, with  $j_0(P)<\cdots<j_n(P)$, called order-sequence at $P$ with respect to $\D$. The integers $j\geq0$ defining this sequence are those  for which  $v_P(D)=j$ for some $D \in \D$, and they are called $(\D,P)$-orders. Since  $\D$ is  base-point-free,   $j_0(P)=0$ for all $P \in \xx$.  When there is no risk of confusion, $j_i$ denotes  $j_i(P)$. Let $L_i(P)$ be the intersection of all hyperplanes in $\p^n(\fqc)$ that intersect $\xx$ at  $P$ with multiplicity at least $j_{i+1}$. The space $L_i(P)$ is called the $i$-th osculating space at $P$. The $(n-2)$-th osculating spaces  will be  of particular interest in  the next section, and so will be the following, as  proved in \cite[Theorem 1.1]{SV}. 

\begin{thm}\label{SV1.1}
Let $P \in \xx$ and $t$ be a local parameter at $P$. Suppose that $e_P=0$. Assume that the first $i$ $(\D,P)$-orders $j_0,\ldots,j_{i-1}$ are known. Then $j_i$ is the smallest integer such that the points $((D_t^{(j_s)}f_0)(P): \cdots:(D_t^{(j_s)}f_n)(P))$ with $s=0,\ldots,i$ are linearly independent, and $L_i(P)$ is spanned by these points.
\end{thm}

All but finitely  many points $P\in \xx$  have the same  order-sequence (see e.g. \cite[Theorem 1.5]{SV}), denoted by $(\epsilon_0,\ldots,\epsilon_n)$, which  is   called the order-sequence of $\xx$ with respect to $\D$. This sequence can also be defined as the minimal sequence, with respect to the lexicographic order, for which
$$
\det\left(D_t^{(\epsilon_i)}f_j\right)_{0 \leq i,j \leq n} \neq 0,
$$
where $t \in \fqc(\xx)$ is a separating variable. Moreover, for each $P \in \xx$,
\begin{equation}\label{ej}
\epsilon_i \leq j_i(P) \text{  for all  } i \in \{0,\ldots,n\}.
\end{equation}
The curve $\xx$ is called classical with respect to $\phi$ (or $\D$) if $(\epsilon_0,\ldots,\epsilon_n)=(0,\ldots,n)$ and  is called nonclassical otherwise.

Now assume that $\phi$ is defined over $\fq$. The sequence of non-negative integers $(\nu_0,\ldots,\nu_{n-1})$, chosen minimally in the lexicographic order, such that
\begin{equation}\label{det fr}
\left|
  \begin{array}{ccc}
  f_0^q & \cdots & f_n^q \\
  D_t^{(\nu_0)}f_0 & \cdots & D_t^{(\nu_0)}f_n \\
   \vdots & \cdots & \vdots \\
  D_t^{(\nu_{n-1})}f_0 & \cdots & D_t^{(\nu_{n-1})}f_n
  \end{array}
  \right| \neq 0,
  \end{equation}
where $t\in \fq(\xx)$ is a separating variable, is called the $\fq$-Frobenius order-sequence of $\xx$ with respect to $\phi$. From \cite[Proposition 2.1]{SV}, we have that $\{\nu_0,\ldots,\nu_{n-1}\}=\{\epsilon_0,\ldots,\epsilon_n\}\backslash\{\epsilon_I\}$ for some $I \in \{1,\ldots,n\}$. If $(\nu_0,\ldots,\nu_{n-1})=(0,\ldots,n-1)$, then  $\xx$ is called $\fq$-Frobenius classical with respect to $\phi$. Otherwise, $\xx$ is called $\fq$-Frobenius nonclassical. The following result follows from  \cite[Theorem 1.1]{Bo2}.

\begin{thm}\label{multi-frob}

Let $m$ and $u$  be coprime  integers, with $m>u\geq 1$, and consider the plane curve  \\
$\mathcal{F}_{m,u}: f(x,y)=0,$ where $f(x,y) \in \F_q[x,y]$ is the  polynomial  $$f(x,y)=\frac{(x^{q^u}-x)(y^{q^m}-y)-(x^{q^m}-x)(y^{q^u}-y)}{(x^{q^2}-x)(y^{q}-y)-(x^{q}-x)(y^{q^2}-y)}.$$
Then  $\mathcal{F}_{m,u}$   is the only  simultaneously $\F_{q^m}$- and  $\F_{q^u}$-Frobenius nonclassical curve for the morphism of lines. It has $\mathcal{F}_{m,u}(\F_q)=\emptyset$ and order-sequence $(0,1,q^u)$.  In particular,  $x \in \F_q(\mathcal{F}_{m,u})$ is a separating variable  and
\begin{equation}\label{multi-frobeq1}
  \left|
  \begin{array}{ccc}
  1 & x^{q^u} & y^{q^u} \\
1 & x & y \\  
D_x^{(q^u)}1 & D_x^{(q^u)}x  & D_x^{(q^u)}y
  \end{array}
  \right|\cdot
  \left|
  \begin{array}{ccc}
  1 & x^{q^m} & y^{q^m} \\
1 & x & y \\  
D_x^{(q^u)}1 & D_x^{(q^u)}x  & D_x^{(q^u)}y
  \end{array}
  \right| \neq 0
  \end{equation}
  
 and   
  
\begin{equation}\label{multi-frobeq2}
\left| 
  \begin{array}{ccc}
  1 & x^{q^u} & y^{q^u} \\
1 & x & y \\  
D_x^{(i)}1 & D_x^{(i)}x & D_x^{(i)}y
  \end{array}
  \right|= 
  \left|
  \begin{array}{ccc}
  1 & x^{q^m} & y^{q^m} \\
1 & x & y \\  
D_x^{(i)}1 & D_x^{(i)}x & D_x^{(i)}y
  \end{array}
  \right|=0
  \text{ for }   i \in \{0,1,\ldots, q^u-1\}.
 \end{equation}
\end{thm}
Based on the proof of \cite[Theorem 1.5]{SV}, one can conclude the following:
\begin{prop}
 Consider $P \in \xx$ and let $t$ be a local parameter at $P$. For each $i \in \{0,\ldots,e\}$, let $z_i \in \fqc(\xx)$ be such that $z_i=t^{j_i}+\sum\limits_{k=j_i+1}^{\infty}a_k t^k$ with $j_0<j_1< \cdots <j_e$, where $e>0$ is an integer. Then, for $0 \leq m_0< \cdots <m_e$, 
\begin{equation}\label{derdet}
v_P\left(\det\left(D_t^{(m_s)}z_r\right)_{ 0 \leq r,s \leq e} \right) \geq \sum\limits_{i=0}^{e}(j_i-m_i),
\end{equation}
and the equality holds if and only if $p \nmid \det\left({j_r \choose  m_s}\right)_{ 0 \leq r,s \leq e}$.
\end{prop}


\section{The $(q^u,q^m)$-Frobenius divisor}\label{The $(q^u,q^m)$-Frobenius divisor}

Hereafter,  $\xx$ will  denote a projective  geometrically irreducible  nonsingular  algebraic curve of genus $g$ defined over $\fq$.
In addition, unless otherwise stated,  the morphisms 
\begin{equation}\label{morfismo}
\phi=(f_0:\ldots:f_n):\xx \lra \p^n(\fqc)
\end{equation}
will always  be  nondegenerated,   defined over $\fq$, and $n\geq 2$.  Let $m$ and $u$ be coprime  integers, with $m>u\geq 1$, and consider  the Frobenius maps  $\Phi_{q^i}: \xx \longrightarrow \xx$, $i\in \{u,m\}$. We want to estimate the number of points $P \in \xx$ for which  a line  through $\Phi_{q^u}(P)$ and $\Phi_{q^m}(P)$ intersects the osculating space  $L_{n-2}(P)$. This and  Theorem \ref{SV1.1} lead us to  study  the functions 
\begin{equation}\label{funçao n-2}
\A_t^{\rho_0,\ldots,\rho_{n-2}}(f_0,\ldots,f_n):=\left|
\begin{array}{cccc}
f_0^{q^m}& f_1^{q^m}&\cdots&f_n^{q^m} \\
f_0^{q^u}& f_1^{q^u}&\cdots&f_n^{q^u}\\
D_t^{(\rho_0)}f_0&D_t^{(\rho_0)}f_1&\cdots&D_t^{(\rho_0)}f_n\\
\vdots & \vdots & \cdots& \vdots \\
D_t^{(\rho_{n-2})}f_0&D_t^{(\rho_{n-2})}f_1&\cdots&D_t^{(\rho_{n-2})}f_n
\end{array}
\right|,
\end{equation}
where $t \in \fq(\xx)$ is a separating variable, and   $\rho_0,...,\rho_{n-2}$ are non-negative integers.

In what follows,  $(\nu_0,\ldots,\nu_{n-1})$ and $(\mu_0,\ldots,\mu_{n-1})$  will be  the $\fqu$- and the $\fqm$-Frobenius order-sequences,  respectively, on $\xx$ associated to   $\D$. 

\begin{prop}\label{seqdf} There exist integers $\kappa_0,\ldots,\kappa_{n-2}$, with   $0\leq \kappa_0<\cdots<\kappa_{n-2}$,
such that\\
 $\A_t^{\kappa_0,\ldots,\kappa_{n-2}}(f_0,\ldots,f_n) \neq 0$. Choose them minimally in the lexicographic order. Then there exist $I, J \in \{0,\ldots, n-1\}$ such that
\begin{equation}\label{kap}
\{\nu_0,\ldots,\nu_{n-1}\}\backslash\{\nu_I\}= \{\kappa_0,\ldots,\kappa_{n-2}\}=\{\mu_0,\ldots,\mu_{n-1}\}\backslash\{\mu_J\}.
\end{equation}
\end{prop}

\begin{proof}
For $i \in \{m,n\}$ and $j\in \{\nu_0,\ldots,\nu_{n-1}\}$,  set $\Phi_{q^i} (\phi):=(f_0^{q^i},\ldots,f_n^{q^i})$ and $D_t^{(j)}(\phi):=(D_t^{(j)}f_0,\ldots,D_t^{(j)}f_n)$ to represent the vectors.  It is clear that $\Phi_{q^u} (\phi)$ and $\Phi_{q^m} (\phi)$ are linearly independent, and  \eqref{det fr} gives the linear independence of  vectors $\Phi_{q^u} (\phi),D_t^{(\nu_{0})}(\phi),\ldots,D_t^{(\nu_{n-1})}(\phi)$. Now if the $n+1$ vectors $$\Phi_{q^m} ( \phi), \Phi_{q^u} (\phi), D_t^{(\nu_{0})}(\phi), \ldots, D_t^{(\nu_{n-2})}(\phi)$$ are linearly independent over $\fq(\xx)$, then the result follows. Otherwise, let $I \in \{0,\ldots,n-2\}$ be the smallest integer for which $\Phi_{q^m} ( \phi)$ is a linear combination of $\Phi_{q^u} (\phi), D_t^{(\nu_{0})}(\phi), \ldots,D_t^{(\nu_{I})}(\phi)$, and  note that the linear independence of  $\Phi_{q^u} (\phi),D_t^{(\nu_{0})}(\phi),\ldots,D_t^{(\nu_{n-1})}(\phi)$ implies that of  
$$\{\Phi_{q^m} (\phi),\Phi_{q^u} (\phi),D_t^{(\nu_{0})}(\phi),\ldots,D_t^{(\nu_{n-1})}(\phi) \}\backslash \{D_t^{(\nu_{I})}(\phi)\}.$$
Therefore,  we clearly have $\{\kappa_0,\ldots,\kappa_{n-2}\}=\{\nu_0,\ldots,\nu_{n-1}\}\backslash\{\nu_I\}$,  and the same argument
implies $\{\kappa_0,\ldots,\kappa_{n-2}\}=\{\mu_0,\ldots,\mu_{n-1}\}\backslash\{\mu_J\}$. 
\end{proof}

Hereafter, the sequence $(\kappa_0,\ldots,\kappa_{n-2})$, given in Proposition \ref{seqdf}, will be called the $(q^u,q^m)$-Frobenius order-sequence of $\xx$ with respect to $\D$ (or  $\phi$). The curve $\xx$ will be called $(q^u,q^m)$-Frobenius classical with respect to $\D$ when $\kappa_i=i$ for all $i \in\{0,1,\ldots,n-2\}$. Otherwise,  $\xx$ will be called $(q^u,q^m)$-Frobenius nonclassical. Note that $\kappa_{n-2} \leq \epsilon_n \leq d=\deg(\D)$.

We shall omit the proof of the following result, as it is analogous to  \cite[Proposition 2.2]{SV}.
\begin{prop}\label{mcp}
\begin{enumerate}[\rm(a)]
\item If $g_i=\sum a_{ij}f_j$ with $(a_{ij}) \in GL_{n+1}(\fq)$, then
\begin{displaymath}
\A_{t}^{\kappa_0,\ldots,\kappa_{n-2}}(g_0,\ldots,g_n)=\det(a_{ij})\cdot \A_{t}^{\kappa_0,\ldots,\kappa_{n-2}}(f_0,\ldots,f_n).
\end{displaymath}
\item If $h \in \fq(\xx)^{*}$, then
\begin{displaymath}
\A_{t}^{\kappa_0,\ldots,\kappa_{n-2}}(hf_0,\ldots,hf_n)=h^{q^m+q^u+n-1}\cdot \A_{t}^{\kappa_0,\ldots,\kappa_{n-2}}(f_0,\ldots,f_n).
\end{displaymath}
\item If $x \in \fq(\xx)$ is another separating variable, then
\begin{displaymath}
\A_{x}^{\kappa_0,\ldots,\kappa_{n-2}}(f_0,\ldots,f_n)=\left(\displaystyle\frac{dt}{dx}\right)^{\kappa_0+\kappa_1+\cdots+\kappa_{n-2}}\cdot \A_{t}^{\kappa_0,\ldots,\kappa_{n-2}}(f_0,\ldots,f_n).
\end{displaymath}
\end{enumerate}
\end{prop}

\begin{defi}\label{T-divisor}
The $(q^u,q^m)$-Frobenius divisor of $\D$ is defined by
\begin{equation}\nonumber
T_{u,m}=\divi (\A_{t}^{\kappa_0,\ldots,\kappa_{n-2}}(f_0,\ldots,f_n))+(\kappa_0+\kappa_1+\cdots+\kappa_{n-2})\divi(dt)+(q^m+q^u+n-1)E,
\end{equation}
where $t \in \fq(\xx)$  is a separating variable and $E=\sum\limits_{P \in \xx} e_P P$  for  $e_P=-min\{v_P(f_0),\ldots,v_P(f_n)\}$.
\end{defi}

Note that given $u$ and $m$,  Proposition \ref{mcp} implies that  the sequence $(\kappa_0,\ldots,\kappa_{n-2})$, as well  as the divisor $T_{u,m}$, depend only on the linear series $\D$. In addition,

\begin{equation}\label{grau t}
\deg(T_{u,m})=(\kappa_0+\kappa_1+\cdots+\kappa_{n-2})(2g-2)+(q^m+q^u+n-1)d,
\end{equation}
where $d:=\deg(\D)=\deg(E)$.

Based on Propostion \ref{mcp} and  Definition \ref{T-divisor}, for any given $P \in \xx$, we can  assume $e_P=0$ before computing $v_P(T_{u,m})$. This will be done systematically,   unless otherwise noted.

\begin{prop}\label{kappa0}
Let $\phi=(f_0:\cdots:f_n):\xx \lra \p^n(\fqc)$ be a morphism defined over $\fq$ and let $(\kappa_0,\ldots,\kappa_{n-2})$ be its $(q^u,q^m)$-Frobenius order-sequence. If $\kappa_0>0$, then $m>n\geq 2$ and the following hold.

\begin{enumerate}[\rm(i)]
\item  The field $\fq(\phi(\xx))$ is $\F_q$-isomorphic to an intermediate field of  $N/\fq(\mathcal{F}_{u,m})$, where  $\mathcal{F}_{u,m}$ is  the curve in Theorem \ref{multi-frob}  and  $N$ is the Galois closure of  $\fq(\ff_{u,m})/\fq(x)$ in a fixed algebraic closure of $\fq(x)$.

\item There exist only  a finite number of curves $\xx$ (up to $\F_q$-isomorphism) that admit a projective model for which $\kappa_0>0$.
\item If $n=2$, then $\phi(\xx)$ is isomorphic to $\ff_{u,m}$.
\item $\kappa_0=q^u$.
\item $\xx(\F_q)=\emptyset$.

\end{enumerate}
\end{prop}

\begin{proof}
Assume $f_0=1$ and set  $x:=f_1$. In particular,  $\fq(\phi(\xx))=\fq(x,f_2,\ldots,f_n)$. Since $\kappa_0>0$, the matrix 
$$
\left(
\begin{array}{ccccc}
1&x^{q^m}&f_2^{q^m} & \cdots & f_n^{q^m} \\
1&x^{q^u}&f_2^{q^u} & \cdots & f_n^{q^u} \\
1&x&f_2 & \cdots & f_n
\end{array}
\right)
$$
has rank $2$, and then, $f_2,\ldots,f_n$ are roots of $(x^{q^m}-x)(T^{q^u}-T)-(T^{q^m}-T)(x^{q^u}-x) \in \F_q(x)[T]$. As  $\phi$ is nondegenerated,  the set $\{a_0+a_1 x+a_2f_2+\cdots +a_nf_n| a_i\in \F_q\}$ has  $q^{n+1}$ distinct roots of  the above polynomial,
which gives $m>n\geq 2$.  In addtion, since $f_2,\ldots,f_n$ are roots of the separable polynomial
$$h(x,T)=\frac{(x^{q^m}-x)(T^{q^u}-T)-(T^{q^m}-T)(x^{q^u}-x)}{(x^{q^2}-x)(T^{q}-T)-(T^{q^2}-T)(x^{q}-x)}\in \fq(x)[T],$$
assertions $(i)$, $(ii)$, and $(iii)$ immediately  follow from Theorem \ref{multi-frob}. To prove $(iv)$, note that $x\in \fq(\phi(\xx))$ is a separating variable, and then  the matrix 
$$
\left(
\begin{array}{ccccc}
1&x^{q^m}&f_2^{q^m} & \cdots & f_n^{q^m} \\
1&x^{q^u}&f_2^{q^u} & \cdots & f_n^{q^u} \\
0&D_x^{(\kappa_0)}x&D_x^{(\kappa_0)}f_2& \cdots & D_x^{(\kappa_0)}f_n
\end{array}
\right)
$$
has rank $3$. Thus for some $y \in\{f_2,\ldots,f_n\}$, we have
\begin{equation}\label{aux pl}
\left|
\begin{array}{ccccc}
1&x^{q^m}&y^{q^m}  \\
1&x^{q^u}&y^{q^u}  \\
0&D_x^{(\kappa_0)}x&D_x^{(\kappa_0)}y
\end{array}
\right|\neq 0, 
\end{equation}
and  the minimality of $\kappa_0>0$ gives
\begin{equation}\label{aux pl1}
\left|
\begin{array}{ccccc}
1&x^{q^m}&y^{q^m}  \\
1&x^{q^u}&y^{q^u}  \\
D_x^{(i)}1&D_x^{(i)}x&D_x^{(i)}y
\end{array}
\right| =0
\text{ for } i\in \{0,\ldots,\kappa_0-1\}.
\end{equation}
In addition, $\kappa_0>0$ implies
\begin{equation}\label{aux pl2}
0=\left|
\begin{array}{ccccc}
1&x^{q^m}&y^{q^m}  \\
1&x^{q^u}&y^{q^u}  \\
D_x^{(i)}1&D_x^{(i)}x&D_x^{(i)}y
\end{array}
\right|
\iff
\left|
\begin{array}{ccccc}
1&x^{q^m}&y^{q^m}  \\
1&x&y \\
D_x^{(i)}1&D_x^{(i)}x&D_x^{(i)}y
\end{array}
\right|= 
\left|
\begin{array}{ccccc}
1&x^{q^u}&y^{q^u}  \\
1&x&y \\
D_x^{(i)}1&D_x^{(i)}x&D_x^{(i)}y
\end{array}
\right|= 0
\end{equation}
for all $i\geq 0$. Therefore,  \eqref{multi-frobeq1} and \eqref{multi-frobeq2} in Theorem \ref{multi-frob}, together with \eqref{aux pl} and \eqref{aux pl2}, imply $\kappa_0=q^u$. 
Finally, note that the case $i=0$ in \eqref{aux pl1} implies $h(x,y)=0$. That is, there exists  a copy of the curve $\mathcal{F}_{u,m}$ being $\F_q$-covered by $\xx$. Since $\ff_{u,m}(\fq)=\emptyset$ (Theorem \ref{multi-frob}), we obtain \textit{(v)}.

\end{proof}

\begin{rem}\label{k0=0}
In view of Proposition \ref{kappa0}, from now on, unless otherwise stated, we always assume that $\kappa_0=0$.
\end{rem}

Let us  establish some  important properties of the $(q^u,q^m)$-Frobenius order-sequence, beginning with the following  result, whose proof is straightforward.

\begin{prop}\label{minimalidade}
If $m_0,\ldots,m_{s}$ are integers with $0 \leq m_0<\cdots<m_s$ such that the rows of the matrix
\begin{equation}\label{minimali}
\left(
\begin{array}{cccc}
f_0^{q^m}& f_1^{q^m}&\cdots&f_n^{q^m} \\
f_0^{q^u}& f_1^{q^u}&\cdots&f_n^{q^u}\\
D_t^{(m_0)}f_0&D_t^{(m_0)}f_1&\cdots&D_t^{(m_0)}f_n\\
\vdots & \vdots &\ddots& \vdots \\
D_t^{(m_{s})}f_0&D_t^{(m_{s})}f_1&\cdots&D_t^{(m_{s})}f_n
\end{array}
\right)
\end{equation}
are linearly independent over $\fq(\xx)$, then $\kappa_i \leq m_i$ for each $i\in\{0,...,s\}$.
\end{prop}

\begin{thm}\label{pdivkappa}
Let $\phi=(f_0:f_1:\cdots:f_n):\xx \lra \p^n(\fqc)$ with $n>2$ be a morphism defined over $\fq$ such that $\xx$ is $(q^u,q^m)$-Frobenius nonclassical with respect to $\phi$. Let $\ell \in \{1,\ldots,n-2\}$ be such that $\kappa_i=i$ for $i<\ell$ and $\kappa_\ell>\ell$. Then $p|\kappa_\ell$. 
\end{thm}
\begin{proof}
Let $t\in \F_q(\xx)$ be a separating variable.  Consider the sequence $(m_0,\ldots,m_{n-2})$ with $m_i=\kappa_i$ for
 $i \in \{0,1,\ldots,n-2\}\backslash \{\ell\}$ and $m_\ell =\kappa_{\ell}-1$. Since $m_\ell< \kappa_{\ell}$, Proposition \ref{minimalidade} implies $\A_t^{m_0,\ldots,m_{n-2}}(f_0,\ldots,f_n)=0.$ Thus it suffices to show that
 \begin{equation}\label{detSn}
 D_{t}^{(1)}\Big(\mathcal{A}_t^{m_0,\ldots,m_{n-2}}(f_0,\ldots,f_n)\Big)=\kappa_{\ell}\cdot \mathcal{A}_t^{(\kappa_0,\ldots,\kappa_{n-2})}(f_0,\ldots,f_n).
 \end{equation}
Let  $S_{n+1}$ denote the group of permutations on  $\{0,1,\ldots, n\}$. We have that
$$\mathcal{A}_{t}^{m_0,\ldots,m_{n-2}}(f_0,\ldots,f_n)=\displaystyle\sum\limits_{\sigma \in S_{n+1}}\Big(\sign(\sigma) \cdot f_{\sigma(0)}^{q^m}\cdot f_{\sigma(1)}^{q^u}\cdot (D_t^{(m_0)} f_{\sigma(2)})  \cdots  (D_t^{(m_{n-2})} f_{\sigma(n)})\Big).$$
Therefore, $D_t^{(1)}\Big(\mathcal{A}_{t}^{m_0,\ldots,m_{n-2}}(f_0,\ldots,f_n)\Big)$ is the sum over all $\sigma \in S_{n+1}$
of   $$\sign(\sigma) \cdot f_{\sigma(0)}^{q^m}\cdot f_{\sigma(1)}^{q^u}\cdot  D_t^{(1)}\Big( (D_t^{(m_0)} f_{\sigma(2)})  \cdots  (D_t^{(m_{n-2})} f_{\sigma(n)})\Big).$$
 In addition, for any fixed $i\in \{0,\ldots,n-2\}$, the sum over all $\sigma \in S_{n+1}$ of the  terms
$$\sign(\sigma) \cdot f_{\sigma(0)}^{q^m}\cdot f_{\sigma(1)}^{q^u}\cdot (D_t^{(m_0)} f_{\sigma(2)})  \cdots \underbrace{\Big(D_t^{(1)}(D_t^{(m_{i})} f_{\sigma(i+2)})\Big)}_{(m_i+1)\cdot D_t^{(m_i+1)} f_{\sigma(i+2)}}\cdots  (D_t^{(m_{n-2})} f_{\sigma(n)})$$
will give rise to the determinant corresponding to $(m_i+1)\cdot\mathcal{A}_{t}^{m_0,\ldots,m_i+1,\ldots, m_{n-2}}(f_0,\ldots,f_n)$. From Proposition \ref{minimalidade}, the latter determinants  vanish if $i\neq \ell$, and we arrive at \eqref{detSn}.
\end{proof}

 The following corollaries are immediate consequences of Theorem \ref{pdivkappa}.

\begin{cor}\label{64-nazar}
 If $p>\deg(\D)$, then $\xx$ is $(q^u,q^m)$-Frobenius classical.
\end{cor}

\begin{cor}\label{65-nazar}
Assume $p>n-1$. If $\xx$ is $(q^u,q^m)$-Frobenius nonclassical w.r.t. $\phi$, then $\xx$ is $\fqr$-Frobenius nonclassical for $r\in\{u,m\}$. Moreover, if $p>n$, then $\xx$ is nonclassical.
\end{cor}


\section{Upper bounds for the number of rational points}

This section provides the main results of this paper. Consider a point $P \in \xx$, and  let  $t$ be a local parameter at $P$.  Since $v_P(\text{div}(dt))=0$ and we  assume  $e_P=0$, it follows that
$$
v_P(T_{u,m})=v_P(\A_{t}^{\kappa_0,\ldots,\kappa_{n-2}}(f_0,\ldots,f_n)) \geq 0.
$$
That is, $T_{u,m}$ is an effective divisor. Moreover, because  any $P \in \xx(\fqr)$ with  $r\in \{u,m,m-u\}$ will render the three first rows of 
$\A_{t}^{\kappa_0,\ldots,\kappa_{n-2}}(f_0,\ldots,f_n)$ 
linearly dependent (cf. Remark \ref{k0=0}), it follows that  $v_P(T_{u,m}) \geq 1$ for  all such points.  The next result  refines  these bounds.
\begin{prop}\label{pesos}
Let $P \in \xx$ with $(\mathcal{D},P)$-orders $j_0, \ldots,j_n$. Then
\begin{numcases}{v_P(T_{u,m}) \geq}
j_1q^u+\displaystyle\sum_{i=0}^{n-2}(j_{i+2}-\kappa_i) \text{ if $P \in \xx(\fq)$}; \label{caso1} \\ 
j_1q^u+\displaystyle\sum_{i=0}^{n-2} (j_{i}-\kappa_i) \text{ for $P \in \xx(\mathbb{F}_{q^{m-u}})$};\label{caso4} \\
\displaystyle\sum_{i=1}^{n-1}(j_i-\kappa_{i-1}), \text{ if $P \in \xx(\mathbb{F}_{q^r})$, for $r \in \{u,m\}$};\label{caso2}\\
\displaystyle\sum_{i=0}^{n-2} (j_{i}-\kappa_i) \text{ for  arbitrary point $P \in \xx$}.\label{caso3}
\end{numcases}
In addition, 
\begin{enumerate}[\rm(i)]
\item Equality  holds in \eqref{caso1} if  and only if $p \nmid \det\left({j_i \choose \kappa_s}\right)_{2 \leq i \leq n, 0 \leq s \leq n-2}$.
\item If $p | \det\left({j_i \choose \kappa_s}\right)_{ 1 \leq i \leq n-1,0 \leq s \leq n-2}$,
then the strict inequality holds in both cases of \eqref{caso2}.
\item If $p | \det\left({j_i \choose \kappa_s}\right)_{ 0 \leq i,s \leq n-2},$
then the strict inequality holds in \eqref{caso4} and \eqref{caso3}
\end{enumerate}
\end{prop}
\begin{proof}
Consider a point $P \in \xx$,  and let  $t$ be a local parameter at $P$. After a suitable projective transformation   $(a_{ij}) \in GL_{n+1}(\fqc)$, we have that $g_i:=\sum\limits_{j=0}^{n}a_{ij}f_i$ has a local expansion $g_i=t^{j_i}+ \cdots$, for $i=0,\ldots,n$, where the dots indicate terms of higher order. Then
\begin{equation}
\mathcal{A}_t^{\kappa_0,\ldots,\kappa_{n-2}}(f_0,\ldots,f_n)\cdot \det(a_{ij})=\left|
\begin{array}{cccc}
b_0&b_1&\cdots&b_n \\
h_0&h_1&\cdots&h_n \\
D_t^{(\kappa_0)}g_0&D_t^{(\kappa_0)}g_1&\cdots&D_t^{(\kappa_0)}g_n \\
D_t^{(\kappa_1)}g_0 & D_t^{(\kappa_1)}g_1 &\cdots& D_t^{(\kappa_1)}g_n \\
\vdots & \vdots & \vdots & \vdots\\
D_t^{(\kappa_{n-2})}g_0 & D_t^{(\kappa_{n-2})}g_1 &\cdots& D_t^{(\kappa_{n-2})}g_n
\end{array}
\right|, \nonumber \\
\end{equation}
where $b_i=\sum\limits_{j=0}^{n}a_{ij}f_j^{q^m}$ and $h_i=\sum\limits_{j=0}^{n}a_{ij}f_j^{q^u}$.
Thus
\begin{equation}\label{At}
\mathcal{A}_t^{\kappa_0,...,\kappa_{n-2}}(f_0,...,f_n)\cdot \det(a_{ij})=\displaystyle\sum_{0 \leq l<k \leq n}(-1)^{l+k+1}\gamma_{lk}\beta_{kl},
\end{equation}
where $\beta_{kl}$ denotes the $(n-1) \times (n-1)$ minor obtained by omitting the first two rows and the $k$-th and $l$-th columns of the matrix above  and $\gamma_{lk}=b_lh_k-h_lb_k$.
From (\ref{derdet}),
\begin{equation}\label{vplki}
v_P(\beta_{kl}) \geq \sum\limits_{\delta=0}^{n-2}(j_\delta-\kappa_\delta)+j_{n-1}+j_n-j_l-j_k, 
\end{equation}
and equality holds if and only if
\begin{equation}\label{ineq}
p \nmid \det\left({j_r \choose \kappa_s}\right)_{r \notin \{l,k\}}.
\end{equation}
In particular,
\begin{equation}\label{vpT}
v_P(T_{u,m}) \geq \min_{0 \leq l<k \leq n}\left\{v_P(\gamma_{lk})+\sum\limits_{\delta=0}^{n-2}(j_\delta-\kappa_\delta)+j_{n-1}+j_n-j_l-j_k\right\}.
\end{equation}
We only  keep track of the pairs $(l,k)$ that will minimize the right hand side of \eqref{vpT}.
Observe that if $P \in \xx(\fqu)$, then we can take  $(a_{ij}) \in GL_{n+1}(\fqc)$ defined over
$\fqu$, and then $h_i=g_i^{q^u}$ for $i \in \{0, \ldots, n\}$. The case $P \in \xx(\fqm)$ is clearly similar. 
In particular, for $P \in \xx(\fq)$, we have $h_i=g_i^{q^u}$ and $b_i=g_i^{q^m}$  for $i \in \{0, \ldots, n\}$.
This observation gives us
\begin{enumerate}
\item[1)] $v_P(\gamma_{lk})\geq 0$ for arbitrary  $P\in \xx$, and then $(l,k)=(n-1,n)$.
\item[2)]  $v_P(\gamma_{lk}) \geq  j_{l}q^u$  for  $P\in \xx$ defined over  $\fqu$ or $\fqm$, and then $(l,k)=(0,n)$.
\item[3)]  $v_P(\gamma_{lk})\geq  j_lq^m+j_kq^u$  for  $P \in \xx(\fq)$, and then $(l,k)=(0,1)$.
\end{enumerate}
Therefore, \eqref{vpT} together with $1)$, $2)$ and $3)$ give the lower bounds in \eqref{caso3}, \eqref{caso2} and  \eqref{caso1}, respectively.

For  $P \in \xx(\mathbb{F}_{q^{m-u}})$, we have  $a_{ij} \in \mathbb{F}_{q^{m-u}}$ for all $i,j$, and then  $a_{ij}^{1/q^m}=a_{ij}^{1/q^u}$. Thus, setting $w_i=\sum\limits_{j=0}^{n}a_{ij}^{1/q^u}f_j$ gives $h_i=w_i^{q^u}$ and $b_i=w_i^{q^m}$, and then $\gamma_{lk}=v_{lk}^{q^u}$, where $v_{lk}:=w_kw_l^{q^{m-u}}-w_lw_k^{q^{m-u}}$. Clearly  $P$ is a zero of $v_{lk}$ with $v_P(v_{lk}) \geq j_1$. Therefore, $v_P(\gamma_{lk}) \geq j_1q^u$  and  \eqref{vpT}  with $(l,k)=(n-1,n)$ gives \eqref{caso4}.

 Finally, one can check that the strict inequality conditions $(i)$, $(ii)$ and $(iii)$ are all given by \eqref{ineq}.
\end{proof}

The proof of the following   result is analogous to that  of \cite[Proposition 2.5]{SV} and omited accordingly.

\begin{prop}\label{mink}
Let $P \in \xx(\fq)$ with $(\D,P)$-orders $j_0,\ldots,j_n$. If $m_0,\ldots,m_{n-2}$ are integers such that $0 \leq m_0 <m_1<\cdots<m_{n-2}$ and 
$$
\det\left({j_i-j_2 \choose m_r}\right)_{0\leq r \leq n-2, \ 2 \leq i \leq n} \not\equiv 0 \mod p,
$$
then $\kappa_i \leq m_i$ for all $i$.
\end{prop}

\begin{cor}\label{pesminrac}
Let $P \in \xx(\fq)$. Then $\kappa_i \leq j_{i+2}-j_2$ for all $i \in \{0,1,\ldots,n-2\}$. Furthermore, $v_P(T_{u,m}) \geq q^uj_1+j_2(n-1).$
\end{cor}
\begin{proof}
The first assertion follows from Porposition \ref{mink} for  $m_i=j_{i+2}-j_2$. The second follows from the first and (\ref{caso1}).
\end{proof}
\begin{thm}\label{mainbound}
Let $\xx$ be a curve of genus $g$ defined over $\fq$. If $\phi:\xx \lra \p^n(\fqc)$ is a morphism defined over $\fq$, with $(q^u,q^m)$-Frobenius order-sequence $(\kappa_0, \kappa_1,\ldots,\kappa_{n-2})$, then
\begin{equation}\label{main}
 c_1N_1+c_u(N_u-N_1)+c_m(N_m-N_1) +c_{m-u}(N_{m-u}-N_1) \leq  (2g-2)\cdot\sum_{i=0}^{n-2}\kappa_i+(q^m+q^u+n-1)d,
\end{equation}
where $d$ is the degree of the linear series $\D$ associated to $\phi$ and $c_r = \min\{v_P(T_{u,m}) \ | \ P \in \xx(\fqr)\}$  for $r\in\{1,u,m, m-u\}$. In addition, $c_{m-u} \geq q^u$ and $c_1 \geq q^u+\epsilon_2(n-1)$.
\end{thm}
\begin{proof}
Since $\xx(\fq) \subseteq \xx(\F_{q^r})$ for $r \in \{u,m,m-u\}$, then $c_1 \geq c_r$ and $N_r \geq N_1$ for each $r$. Moreover, since $\gcd(u,m)=1$, we have $\xx(\F_{q^r})\bigcap\limits_{r \neq s}\xx(\F_{q^s})=\xx(\fq)$, for $r,s \in \{1,u,m,m-u\}$. Thus (\ref{main}) follows from (\ref{grau t}). Corollary \ref{pesminrac} and (\ref{caso4})   imply the last statement.
\end{proof}

\begin{rem}
One can subtract $B(P):=v_P(T_{u,m})-c_r$ (resp. $v_P(T_{u,m})$) on the right side of \eqref{main} for each $P \in \xx(\fqr)$ (resp. for each $P \in \xx$), with the values of $v_P(T_{u,m})$ estimated via Proposition \eqref{pesos}.
\end{rem}

\begin{cor}\label{dfclass}
Assume that $\xx$ is $(q^u,q^m)$-Frobenius classical with respect to $\phi$. Then
\begin{equation}\label{cotadfc}
(n-1)N_u+(n-1)N_m +q^uN_{m-u}  \leq  (n-1)(n-2)(g-1)+d(q^m+q^u+n-1)-\sum_{P \in \xx(\fq)}B(P),
\end{equation}
where $d$ is the degree of the linear series associated to $\phi$ and $B(P) \geq q^u(j_1-1)+\sum\limits_{i=2}^{n}(j_i-i)$.
\end{cor}

Note that  for $n=m=2$, Proposition \ref{kappa0} implies $\kappa_0=0$, i.e., the curve is $(q,q^2)$-Frobenius classical. Thus  Corollary \ref{dfclass} further gives

\begin{cor}\label{dfclass-p}
Let $\mathcal{C}$ be an irreducible plane curve  of degree $d>1$ defined over $\F_q$. If $N_i=\mathcal{C}(\F_{q^i})$, $i=1,2$, then
\begin{equation}\label{cotadfc-p}
(q+1)N_1+N_2  \leq d(q^2+q+1)-\sum_{P \in \xx(\fq)}(q^u(j_1-1)+(j_2-2)).
\end{equation}
\end{cor}

Under suitable conditions, one can improve the estimate of Proposition \ref{pesos}.

\begin{prop}\label{wfnc}
 Let $(\nu_i)$ and $(\mu_i)$ be the respective $\F_{q^u}$- and  $\F_{q^m}$-Frobenius order-sequences  with  $\mu_i=\epsilon_i$,
 $i=0,\ldots,n-1$ and   $\{\nu_0,\nu_1,\ldots,\nu_{n-1}\}=\{\epsilon_0,\epsilon_1,\ldots,\epsilon_{n}\}\backslash \{\epsilon_k\}$ for some $k\in \{1,\ldots,n-1\}$. If $P\in \xx(\F_{q^m})$ has $(\mathcal{D},P)$-orders $j_0,j_1,\ldots,j_n$ and  $p\nmid \det \Big(\binom{j_i}{\epsilon_r}\Big)_{0\leq i,r\leq n-1}$, then
 \begin{equation}\label{Frob peso}
 v_P(T_{u,m})\geq j_n +\sum_{\substack{i=1 \\ i\neq k}} ^{n-1}(j_i-\epsilon_i).
 \end{equation}
 \end{prop}

 \begin{proof}
 
For the sake of notation simplicity, the proof will be limited to the case $k=n-1$. The general case is  analogous.
Let $t\in\F_q(\xx)$ be  local paramter  at $P$. Thus we  assume that the coordinate  functions $f_i \in \F_q(\xx)$ are regular at $P \in \xx(\F_{q^m})$ and  that $f_0=1$. Since  $\nu_{n-1}=\epsilon_n$, we have
\begin{equation}\label{det fnc}
\left|
\begin{array}{cccc}
1 & f_1^{q^u} & \cdots & f_n^{q^u}\\
1 & f_1 & \cdots&f_n\\
0 & D_t^{(\epsilon_1)}f_1&\cdots&D_t^{(\epsilon_1)}f_n\\
\vdots & \vdots & \ddots & \vdots \\
0 & D_t^{(\epsilon_{n-1})}f_1& \cdots&D_t^{(\epsilon_{n-1})}f_n
\end{array}
\right| =0,
\end{equation}
and then 
\begin{equation}\label{CombLin}
(1, f_1^{q^u},  \cdots, f_n^{q^u})=\sum\limits_{i=0}^{n-1}\delta_i\cdot \Big(D_t^{(\epsilon_{i})}(1), D_t^{(\epsilon_{i})}(f_1),  \cdots, D_t^{(\epsilon_{i})}(f_n)\Big) 
\end{equation}
for some $\delta_0,\delta_1,\ldots,\delta_{n-1} \in \F_q(\xx)$.  From Proposition \ref{seqdf}, we have
$\{\kappa_0,\kappa_1,\ldots, \kappa_{n-2}\}=\{\epsilon_0,\epsilon_1,\ldots,\epsilon_{n-2}\}$, and  from
\begin{equation}\label{det fnc}
0\neq \left|
\begin{array}{cccc}
1 & f_1^{q^m} & \cdots & f_n^{q^m}\\
1 & f_1^{q^u} & \cdots & f_n^{q^u}\\
1 & f_1 & \cdots&f_n\\
0 & D_t^{(\epsilon_1)}f_1&\cdots&D_t^{(\epsilon_1)}f_n\\
\vdots & \vdots & \ddots & \vdots \\
0 & D_t^{(\epsilon_{n-2})}f_1& \cdots&D_t^{(\epsilon_{n-2})}f_n
\end{array}
\right| =\pm
\delta_{n-1}\cdot \left|
\begin{array}{cccc}
1 & f_1^{q^m} & \cdots & f_n^{q^m}\\
1 & f_1 & \cdots&f_n\\
0 & D_t^{(\epsilon_1)}f_1&\cdots&D_t^{(\epsilon_1)}f_n\\
\vdots & \vdots & \ddots & \vdots \\
0 & D_t^{(\epsilon_{n-1})}f_1& \cdots&D_t^{(\epsilon_{n-1})}f_n
\end{array}
\right|,
\end{equation}
we have   $v_P(T_{u,m})=v_P(\delta_{n-1})+v_P(S)$, where $S$ is the St\"ohr-Voloch  divisor with respect to  $\F_{q^m}$. Hence 
\cite[Proposition 2.4 (a)]{SV}   gives 
\begin{equation}\label{passo1}
 v_P(T_{u,m})\geq v_P(\delta_{n-1})+\sum\limits_{i=1}^{n} (j_i-\epsilon_{i-1}).
\end{equation}
Let us proceed to evaluate $v_P(\delta_{n-1})$. Note that if $\sum\limits_{i=0}^n a_iX_i=0$ is the osculating hyperplane at our point $P \in \xx(\F_{q^m})$, then  $v_P(\sum\limits_{i=0}^n a_if_i)=j_n$ gives $a_i\neq 0$ for some $i\geq 1$. Without loss of generality, we  assume $a_n\neq 0$.  Thus the  uniqueness of the osculating hyperplane  implies that $j_0,\ldots,j_{n-1}$ are also the $(\mathcal{D}^{\prime},P)$-orders with respect to   $\xx \xrightarrow{(1:f_1:\ldots:f_{n-1})} \mathbb{P}^{n-1}$. Removing the last coordinate on the vectors of \eqref{CombLin},  we write 
\begin{equation}\label{det fnc}
\left[
\begin{array}{cccc}
1 & 0 &\cdots& 0\\
f_1 & D_t^{(\epsilon_1)}f_1&\cdots& D_t^{(\epsilon_{n-1})}f_1\\
\vdots & \vdots & \ddots &  \vdots \\
f_{n-1}&D_t^{(\epsilon_1)}f_{n-1} & \cdots& D_t^{(\epsilon_{n-1})}f_{n-1}\\
\end{array}
\right]
\left[
\begin{array}{cccc}
\delta_0\\
\delta_1\\
\vdots  \\
\delta_{n-1}
\end{array}
\right]=
\left[
\begin{array}{cccc}
1\\
f_1^{q^u}\\
\vdots  \\
f_{n-1}^{q^u}
\end{array}
\right].
\end{equation}
The function $\delta_{n-1}$  will be obtained via  Cramer's rule in \eqref{det fnc}. To this end, we consider the following   two  determinants associated to $\xx \xrightarrow{(1:f_1:\ldots:f_{n-1})} \mathbb{P}^{n-1}$:
\begin{equation}
A=\left|
\begin{array}{cccc}
1 & f_1 & \cdots&f_{n-1}\\
0 & D_t^{(\epsilon_1)}f_1&\cdots&D_t^{(\epsilon_1)}f_{n-1}\\
0 & D_t^{(\epsilon_2)}f_1&\cdots&D_t^{(\epsilon_2)}f_{n-1}\\
\vdots & \vdots & \ddots & \vdots \\
0 & D_t^{(\epsilon_{n-1})}f_1& \cdots&D_t^{(\epsilon_{n-1})}f_{n-1}
\end{array}
\right|
\text{ and }
B=\left|
\begin{array}{cccc}
1 & f_1^{q^u} & \cdots&f_{n-1}^{q^u}\\
1 & f_1 & \cdots&f_{n-1}\\
0 & D_t^{(\epsilon_1)}f_1&\cdots&D_t^{(\epsilon_1)}f_{n-1}\\
\vdots & \vdots & \ddots & \vdots \\
0 & D_t^{(\epsilon_{n-2})}f_1& \cdots&D_t^{(\epsilon_{n-2})}f_{n-1}
\end{array}
\right|.
\end{equation}
The usual computations lead us to $A=\det \Big(\binom{j_i}{\epsilon_r}\Big)\cdot t^{\sum\limits_{i=0}^{n-1}(j_i-\epsilon_i)}+\cdots$ and 
$v_P(B)\geq \sum\limits_{i=0}^{n-2}(j_i-\epsilon_i)$.  Our  hypothesis  $\det \Big(\binom{j_i}{\epsilon_r}\Big)\neq0$ further gives
$v_P(A)=\sum\limits_{i=0}^{n-1}(j_i-\epsilon_i)$. Hence Cramer's rule  yields
$$v_P(\delta_{n-1})=v_P(B/A)\geq \sum\limits_{i=0}^{n-2}(j_i-\epsilon_i)-\sum\limits_{i=0}^{n-1}(j_i-\epsilon_i)=-(j_{n-1}-\epsilon_{n-1})$$  which, together with \eqref{passo1}, gives the result.
 \end{proof}

\begin{rem}
Note that under the conditions of Proposition \ref{wfnc}, we have that $\epsilon_{k+1}=p^r$ for some $r>0$.
\end{rem}

 \begin{prop}\label{pbfnc}
Consider $\xx$ with the hypotheses of Proposition \ref{wfnc} such that $p\nmid \det \Big(\binom{j_i(P)}{\epsilon_r}\Big)_{0\leq i,r\leq n-1}$ for all $P \in \xx(\F_{q^m})\backslash \xx(\fq)$. Then
\begin{equation}\label{bfnc}
\epsilon_kN_u+\epsilon_nN_m+q^uN_{m-u} \leq  (2g-2)\cdot \sum_{\substack{i=1 \\ i\neq k}} ^{n-1}\epsilon_i +d(q^m+q^u+n-1),
\end{equation}
where $d=\deg(\D)$.
\end{prop}
\begin{proof}
The hypotheses of Proposition \ref{wfnc} provide $\{\kappa_0,\ldots,\kappa_{n-2}\}=\{\epsilon_0,\ldots,\epsilon_{n-1}\}\backslash\{\epsilon_k\}$ and $\{\nu_0, \ldots, \nu_{n-1}\}=\{\kappa_0,\ldots,\kappa_{n-2}\} \cup\{\epsilon_n\}$. In particular, $\epsilon_i \geq \kappa_{i-1}$ for all $i \in \{1,\ldots,n-1\}$ and $\epsilon_n=\nu_{n-1}$. Let $P \in \xx(\F_{q^u})$. Then  Proposition \ref{pesos} implies that
$$
v_P(T_{n,m}) \geq \sum_{i=1}^{n-1}(j_i(P)-\kappa_{i-1}) \geq \epsilon_k.
$$
Moreover, by \cite[Corollary 2.6]{SV}, we have $j_n(P) \geq \epsilon_n+j_1(P)$, and then
$$
v_P(T_{n,m}) \geq q^uj_1(P)+ \sum_{i=0}^{n-2}(j_{i+2}(P)-\kappa_{i}) \geq q^uj_1(P)+\epsilon_n+j_1(P)+ \sum_{i=2}^{n-1}(j_{i}(P)-\kappa_{i-1})
$$
for all $P \in \xx(\fq)$. If $P \in \xx(\F_{q^{m-u}})$, Proposition \ref{pesos} gives $v_P(T_{u,m}) \geq q^uj_1(P)$, and if $P \in \xx(\F_{q^{m}})$, then $v_P(T_{u,m}) \geq j_n(P)$ by Proposition \ref{wfnc}. Thus (\ref{bfnc}) follows from Theorem \ref{mainbound}.
\end{proof}

\begin{cor}\label{plainfnc}
Let $\ff$ be an irreducible, $\F_{q^u}$-Frobenius nonclassical  plane curve of degree $d$. If $j_1(P) \not\equiv 0 \mod p$ for all $P \in \xx(\F_{q^m})\backslash \xx(\fq)$, then
\begin{equation}\label{bfncp}
\epsilon_2N_m+q^uN_{m-u} \leq (q^m+d-1)d.
\end{equation}
\end{cor}
\begin{proof}
Since $\ff$ is $\F_{q^u}$-Frobenius nonclassical, it follows from \cite[Corollary 1.4]{BH} that $N_u \geq d(q^u-d+2)$. Hence we obtain (\ref{bfncp}) directly from Proposition \ref{pbfnc}.
\end{proof}

\begin{rem}
As in Theorem \ref{mainbound}, bound \eqref{bfnc} can be improved according to the values $B(P)$, $P \in \xx$.
\end{rem}

Assuming $m=2$ and taking into account the values $B(P)$ for $P \in \xx(\fq)$, we obtain the following. 

\begin{prop}\label{FNC-p}
Consider $\xx$ with the hypotheses of Proposition \ref{wfnc} with $m=2$. Assume that that $p\nmid \det \Big(\binom{j_i(P)}{\epsilon_r}\Big)_{0\leq i,r\leq n-1}$ for all $P \in \xx(\F_{q^2})\backslash \xx(\fq)$. Then
\begin{equation}
(q+\epsilon_k)N_1+\epsilon_nN_2 \leq (2g-2) \cdot \sum_{\substack{i=1 \\ i\neq k}} ^{n-1}\epsilon_i +d(q^2+q+n-1)-\displaystyle\sum_{P \in \xx(\mathbb{F}_{q})}B(P),
\end{equation}
where $d=\deg(\D)$ and $B(P)\geq q(j_1(P)-1)-j_1(P)+\sum\limits_{i=1}^{n}(j_i(P)-\epsilon_i)$. In particular, if $n=2$,
\begin{equation}
(q+1)N_1+\epsilon_2N_2 \leq d(q^2+q+1)-\displaystyle\sum_{P \in \xx(\mathbb{F}_{q})}\Big(q\big(j_1(P)-1\big)+j_2(P)-\epsilon_2-j_1(P)\Big).
\end{equation}
\end{prop}

\section{Examples}\label{examples}

The effectiveness of the bounds  presented  in the  previous sections  can be verified in many cases. This  section illustrates  cases found in the literature. 

\begin{ex}

For irreducible  plane curves of degree $d\geq 2$, defined over $\F_q$, Corollary \ref{dfclass-p} gives
\begin{equation}\label{simplest}
(q+1)N_1+N_2\leq (q^2+q+1)d-\sum\limits_{P \in \xx(\F_q)}(j_2(P)-2).
\end{equation}
This bound is clearly sharp for  conics, but one  can  certainly  find  less obvious cases of sharpness. For instance, for $q\equiv 0 \mod 3$, the Fermat  curve over $\F_{q}$    
$$\xx: x^{q-1}+y^{q-1}+z^{q-1}=0$$
has no $\F_q$-inflection point;   $N_1=({q-1})^2$ and  $N_2=3(q-1)+(q-1)^2$ \textup{(see  (\cite[Theorem 1]{Mo})}. 
 \end{ex}
\begin{ex}
Bound \eqref{simplest} may be of interest, even if we focus on $N_1$. In fact, since $N_1\leq N_2$, we have
\begin{equation}\label{Hoki} 
N_1\leq (q-1+\frac{3}{q+2})d.
\end{equation}
Curiously, a  bound  similar to  \eqref{Hoki} was proved by Homma and Kim in \textup{\cite{HoKi}}\footnote{Let $f(x,y,z) \in \fq[x,y,z]$ be a homogeneous polynomial of degree $d$ without $\fq$-linear factor. In \cite{Ta}, Tallini proved that if $R_q \geq q^2+q+1$, then $d \geq q+2$, where $R_q$ denotes the number of solutions of $f(x,y,z)=0$ in $\p^2(\fq)$. It follows from bound (\ref{Hoki}) that the same result holds when we replace $R_q$ by $N_1$.}. However,  in their context,  $N_1$ is  the number  of points $P\in \mathbb{P}^2(\F_q)$ on a plane curve defined over $\F_q$,  without  $\F_{q}$-linear components. Therefore, in the singular setting, the two  bounds have different meanings.  Examples of curves attaing  \eqref{Hoki} can be easily found in  the literature.  For instance, the smooth curves over  $\F_q$
$$y(y^qz-yz^q)+z(z^qx-zx^q) + (ax+by+cz)(x^qy-xy^q)=0,$$
where $a,b,c \in \F_{q}$ are such that $t^3-(ct^2+bt+a) \in \F_q[t]$  is irreducible \textup{(see \cite{HoKi1})}. 

Bound  \eqref{Hoki} can be  slightly improved if  rational inflection points occur. For instance, suppose  $\xx$ is a smooth plane curve of degree $d$ and $P\in \xx(\F_q)$  is a total inflextion point, that is, $j_2(P)=d$. Thus  \eqref{simplest}  reads    $(q+1)N_1+N_2\leq (q^2+q+1)d-(d-2),$
and then $N_1\leq N_2$ gives 
\begin{equation}\label{FlexBound}
N_1\leq \frac{(q^2+q)d+2}{q+2}.
\end{equation}                                      
The  smooth plane curve  $ x^{q+1}-x^2z^{q-1}+x^qz-yz^q=0$ over $\F_q$ has  a total inflection point, namely $P=(0 : 1 : 0)$,
and attains  bound \eqref{FlexBound}, as $N_1=q^2+1$ \textup{(see \cite{HoKi})}.

 In the  latter examples, tantamount to attaining the upper bound for $N_1$ is concluding via \eqref{simplest}  that  $N_2=N_1$. These  small values for $N_2$ could not be predicted by  classical bounds such as Weil's, St\"ohr-Voloch's, or Ihara's  bound in \textup{\cite{Iha}}. 
 \end{ex}
 \begin{ex} One  can  also take advantage of the bounds if there is  prior knowledge of, say, $N_1$. 
Some  Frobenius nonclassical curves will be used to illustrate  this. If  $\mathcal{C}$ is a plane $\F_{q}$-Frobenius nonclassical curve
then, from Theorem \ref{multi-frob},  $\mathcal{C}$  is  $\F_{q^2}$-Frobenius classical. Thus if  $p\nmid j_1(P)$
for all $P \in \mathcal{C} (\F_{q^2}) \backslash \mathcal{C} (\F_{q})$, then Proposition \ref{FNC-p} gives 

\begin{equation}\label{bound ex1}
(q+1)N_1+\epsilon_2N_2 \leq d(q^2+q+1)-\displaystyle\sum_{P \in \xx(\mathbb{F}_{q})}\Big(q\big(j_1(P)-1\big)+j_2(P)-\epsilon_2-j_1(P)\Big).
\end{equation}
For the $\F_{q^3}$-Frobenius nonclassical  curve  \textup{(see \cite{Bo0},\cite{GV2},\cite{HV})}
 $$\mathcal{Y}_{q,3}: x^{q^2+q+1}+y^{q^2+q+1}=1,$$ 
 it is known that   $N_1:=\#\mathcal{Y}_{q,3}(\F_{q^3})=q^5 -q^3 -q^2 + 1$ and $\epsilon_2=q$. The  $3(q^2+q+1)$ inflection points $P_i \in \mathcal{F}(\F_{q^3})$  have order-sequence $(j_0,j_1,j_2)=(0,1,q^2+q+1)$, and then $B(P_i) \geq q^2$. Using these data,  bound \eqref{bound ex1}  gives  $N_2\leq q^6 + 2q^5 + q^4 - 2q^3 - 2q^2 - q + 1$.
Similarly, one can bound $N_2:=\#\mathcal{F}(\F_{q^6})$ for the Norm-Trace curve  $$\xx_{q,3}: x^{q^2+q+1}=y^{q^2}+y^q+y,$$ and arrive at $N_2\leq q^6+q^5-q^3+1$. It can be checked that these bounds are better than Weil's, St\"ohr-Voloch's and Ihara's bounds. In fact, it is not   difficult  to prove that these  are the actual values for $N_2$.
 \end{ex}
 \begin{ex} For an example where $\xx\hookrightarrow \mathbb{P}^n$ and $n>2$, consider the Fermat curve 
\begin{equation}\label{FermatConica}
\mathcal{F}: x^{\frac{q+1}{2}}+y^{\frac{q+1}{2}}=1
\end{equation} over $\F_{q^2}$   with $q\not \equiv 0 \mod 5$. This  curve is a well-known $\F_{q^2}$-maximal curve of genus $g=\frac{(q-1)(q-3)}{8}$.
In particular, $N_1=1+q^2+2gq$ and $N_2=1+q^4-2gq^2$. In addition,  $\mathcal{F}$ is $\F_{q^2}$-Frobenius nonclassical for the morphism of conics $\mathcal{F} \hookrightarrow \mathbb{P}^5$ \textup{(see \cite{GV2})}. Simple  manipulation  of equation \eqref{FermatConica} will give us that 
$\mathcal{C}_P: g_P(x,y)=0$, where
$$g_P(x,y)=(  a^{q}\cdot x+b^{q}\cdot y -1)^2 -4(ab)^{q}xy,$$
is the osculating conic at a general point $P=(a,b) \in \mathcal{F}$.  Since  $\mathcal{F}$ is classical for the morphism of lines, it  follows that $\epsilon_i=i$ for $i\leq 4$ and $\epsilon_5=q$.  While the $\F_{q^2}$-Frobenius map clearly takes $P$ to $\mathcal{C}_P$,  i.e., $g_P(a^{q^2},b^{q^2})=0$, this is not  the case  for the $\F_{q^4}$-Frobenius map, that is, the function 
$$( x^{q}\cdot x^{q^4}+y^{q}\cdot y^{q^4} -1)^2 -4(xy)^{q}(xy)^{q^4}$$
does not vanish. Thus  $\mathcal{F}$ is $\F_{q^4}$-Frobenius  classical. In additon, since all  inflection points lie in $\mathcal{F}(\F_{q^2})$,  it follows that    $\det \Big(\binom{j_i(P)}{\epsilon_r}\Big)_{0\leq i,r\leq 4}=1$ for all $P\in \mathcal{F}(\F_{q^4})\backslash \mathcal{F}(\F_{q^2})$. Therefore,  Proposition \ref{FNC-p} gives
\begin{equation}\label{bound example1}
(q^2+\epsilon_4)N_1+\epsilon_5N_2 \leq (2g-2) \cdot \sum_{\substack{i=1}} ^{3}\epsilon_i +d(q^4+q^2+4)-\displaystyle\sum_{P \in \xx(\mathbb{F}_{q^2})}B(P)
\end{equation}
Note that each of the  $3(q+1)/2$  inflection points $P_i$ have order-sequence $(0,1,2,(q+1)/2,(q+3)/2,q+1)$,
and then $B(P_i)\geq q-5$. Plugging the corresponding  values  in  \eqref{bound example1}, we see that  it becames an equality.
\end{ex}
\begin{ex}\label{atinge a cota df retas}

To conclude examples, consider the Hermitian curve
$$
\mathcal{H}: x^{q+1}+y^{q+1}=1
$$
 over $\F_{q^2}$.  The curve  $\mathcal{H}$ is $\F_{q^2}$-Frobenius nonclassical with respect to the morphism  of lines, and $\epsilon_2=q$. Since    $\mathcal{H}$ is $\F_{q^2}$-maximal of genus $g=q(q-1)/2$, we have $N_m=q^{2m}+1+(-1)^{m-1}q^{m+1}(q-1)$ \textup{(see e.g. \cite[p.336]{HKT})}.  Thus  
 \begin{equation}
qN_m+q^2N_{m-1}  = d(q^{2m}+d-1), 
\end{equation}
  where $d=q+1$ is the degree of $\mathcal{H}$. That is, $\mathcal{H}$  meets  the bound in Corollary \eqref{plainfnc}.

\end{ex}

\subsection*{Acknowledgments}
The first author was partially supported by FAPESP-Brazil, grant 2013/00564-1.


\printindex

\end{document}